\documentclass{amsart}
\usepackage{bbm}
\usepackage{amsmath,amsthm,amsfonts,amssymb,amscd,mathrsfs}
\usepackage{psfrag,graphicx}
\usepackage{epic,eepic}
\usepackage{color}
\usepackage[all]{xy}
\usepackage{ebezier}
\usepackage{enumerate}

\newtheorem{Thm}{Theorem}[section]
\newtheorem{Lem}[Thm]{Lemma}
\newtheorem{Prop}[Thm]{Proposition}
\newtheorem{Cor}[Thm]{Corollary}
\theoremstyle{remark}
\newtheorem{Def}[Thm] {Definition}

\newtheorem{Con}{Conjecture}

\newcommand{\length}{\operatorname{\length}}

\newcommand{\Diff}{\operatorname{Diff}}

\def\Diff{\operatorname{Diff}}

\def\length{\operatorname{length}}

\title[Density of Non-uniform hyperbolicity diffeomorphisms]{The $C^1$ density of nonuniform hyperbolicity in $C^{ r}$ conservative diffeomorphisms}
\author
{Chao Liang, Yun Yang}

 \address{C.Liang, Applied Mathematical Department, Central University of Finance
and Economics, Beijing, 100081, China}
 \email{chaol@cufe.edu.cn} 
 
 \address{Y.Yang, School of Mathematical Sciences,
Peking University, Beijing, 100871, China}
\email{yangy88@pku.edu.cn}

\begin{document}

\def\abstractname{\textbf{Abstract}}
\begin{abstract}
Let $\Diff^{ r}_m(M)$ be the set of $C^{ r}$ volume-preserving  diffeomorphisms on a compact Riemannian manifold $M$ ($\dim M\geq 2$).  In this paper, we prove that the diffeomorphisms without zero Lyapunov exponents on a set of positive volume are $C^1$ dense in $\Diff^{ r}_m(M), r\geq 1$.   We also prove a weaker result for symplectic diffeomorphisms  $\mathcal{S}ym^{r}_{\omega}(M), r\geq1 $ saying that the symplectic diffeomorphisms  with non-zero Lyapunov exponents on a set of positive volume are $C^1$ dense in $\mathcal{S}ym^{r}_{\omega}(M), r\geq1 $.

 \end{abstract}

\maketitle
\section{Introduction}

The hyperbolic behavior is of high  importance  to detect the complexity of a dynamical system. The understanding of  hyperbolic behavior, especially the prevalence of hyperbolic behavior, attracts one's attention for a long time. Mainly, there are three classes of systems with hyperbolicity: uniformly hyperbolic systems, partially hyperbolic systems and non-uniformly hyperbolic systems.  Uniformly hyperbolic diffeomorphisms and partially hyperbolic diffeomorphisms are both  $C^r (r\geq 1)$ robust and thus form an open subset  in the space of $C^r$ diffeomorphisms. However, there are some topological restrictions on the underlying manifold $M$ for the existence of a uniformly hyperbolic diffeomorphism, and also of a partially hyperbolic diffeomorphism. Global rigidity of Anosov actions, namely, the classification of Anosov systems is one of the most  striking  problems in dynamics, see e.g. \cite{KS,HZ}.  It is conjectured that Anosov diffeomorphisms are always of algebraic nature, up to topological conjugacy.
Also for partially hyperbolic systems, it is known that a $3$-dimensional sphere does not admit a partially hyperbolic diffeomorphism (\cite{BBI}). Thus there are many compact Riemannian manifolds which do not admit any uniformly hyperbolic diffeomorphisms or even any partially hyperbolic diffeomorphisms. Nevertheless, we are happy to know  that there are no topological obstructions on the manifolds for  the existence of a non-uniformly hyperbolic diffeomorphism.
 \begin{Thm}(\cite{DP})\label{DP}Given a compact smooth Riemannian manifold $K\neq S^1$ there exists a $C^{\infty}$ diffeomorphism $f$ of $K$ such that:
  \begin{enumerate}
  \item $f$ preserves the Riemannian volume $m$ on $K$;
  \item $f$ has nozero Lyapunov exponents at $m-$almost every point $x\in K$;
  \item $f$ is a Bernoulli diffeomorphism.
  \end{enumerate}
  \end{Thm}


As mentioned above, one may care not only about the existence but also about the prevalence of nonuniformly hyperbolic systems. Usually, when we refer to nonuniform hyperbolicity, we associate it with some specific measure. Pesin questioned that whether the systems with non-uniform hyperbolicity behavior  are dense  in $C^r$ conservative diffeomorphisms ($r>1$) and he formulated this question by means of the following conjecture.   
\begin{Con}[Pesin]\label{Pesin}\cite{P,CP}Let $f$ be a $C^r$ conservative diffeomorphism ($r>1$) of a compact smooth Riemannian manifold $M$ ($\dim M\geq 2$). Then arbitrarily close to $f$ in $\Diff^r_m(M)$  (where $m$ is volume), there is a diffeomorphism $g\in \Diff^r_m(M)$ without zero Lyapunov exponents on a set of positive volume. 
\end{Con}
Motivated by this conjecture, we prove in this paper that non-uniformly hyperbolic behavior are $C^1$ dense in $C^{ r}, r\geq 1$ conservative diffeomorphisms.  It is worth to note that the density in $r>1$ regularity is still widely open.

 \begin{Thm}\label{Main}Let $f$ be a $C^{ r}$  conservative diffeomorphism ($r\geq 1$) on a compact smooth  Riemannian manifold $M$ ($\dim M\geq 2$). Then $C^1$ arbitrarily close to $f$ in $\Diff^{ r}_m(M)$  (where $m$ is volume), there is a diffeomorphism $g\in\Diff^{ r}_m(M)$ without zero Lyapunov exponents on a set of positive volume.  \end{Thm}
 \begin{Cor}\label{Density} Let  $\Diff^+_m(M)$ be the set of $C^{ r}$ conservative diffeomorphisms ($r\geq1$) on  a compact smooth  Riemannian manifold $M$ ($\dim M\geq 2$) with positive metric entropy with respect to volume measure $m$. Then, we have $$\overline{\Diff^+_m(M)}=\Diff^1_m(M).$$
 \end{Cor}

In  contrast with the result of Bochi \cite{B} saying that for a  residual subset $R$ of $C^1$ conservative diffeomorphisms on surface, either the diffeomorphism is Anosov or the  Lyapunov exponents vanish almost everywhere. Theorem \ref{Main}  illustrates clearly  that a generic subset may be totally different from  a dense subset.   It is also interesting to compare our theorem with the results in \cite{CS} and \cite{X}. They show that on any manifold $M$ of dimension at least two, there are open sets of volume-preserving diffeomorphisms  with high regularity, all of which have positive measure invariant tori and all of the Lyapunov exponents are zero on these tori. Thus, due to the existence of  these invariant tori, one can't expect the nonuniform hyperbolicity behavior in Conjecture \ref{Pesin} to be global when $r$ is large enough.  
  Another related result for the $SL(2,\mathbb{R})$-cocycles  was  addressed in Avila's paper \cite{A}.   He considered the density of non-uniform hyperbolicity in the considerably simpler context of $SL(2,\mathbb{R})$-cocycles. It is noticeable that he treated not only arbitrary underlying dynamics but also all usual regularity classes. 

Now let's say something about the main idea of the proof for Theorem \ref{Main}.  To solve Pesin's conjecture, it is of utmost importance to eliminate the zero exponents by perturbations. 
Shub-Wilkinson's example\cite{SW} builds a conservative
perturbation to a skew product of an Anosov diffeomorphism of the
torus $T^2$ by rotations and creates positive exponents in
the center direction for Lebesgue almost every point.
Baraviera-Bonatti\cite{BB} present a local version of Shub-Wilkinson's
argument, allowing one to remove 
zero-integrated Lyapunov exponents of any conservative partially
hyperbolic systems. 
Recently, Avila, Crovisier and Wilkinson \cite{ACW} succeed  in eliminating zero Lyapunov exponents of diffeomorphisms without global dominated splitting   and proved a striking result about the relation between the systems with robust positive metric entropy and the systems with  global dominated splitting: 
\begin{Thm}[Theorem A, \cite{ACW}]\label{ACWT}A $C^1$ generic map $f\in\Diff^1_m(M)$ with positive metric entropy is nonuniformly Anosov and ergodic.
\end{Thm}
\begin{Cor}[Corollary, \cite{ACW}]\label{ACW} A map $f\in\Diff^1_m(M)$ has robust positive metric entropy if and only if it admits a dominated splitting. 
\end{Cor} 
 Applying Theorem \ref{ACWT} and  the robustness of Pesin block (formulated in Sublemma 5.1 in \cite{AB}, or see Lemma \ref{Robust block} in this paper), we finish the proof for the maps with dominated splitting. Thus it suffices to approximate  the systems with no global domination by the systems with  nonuniform hyperbolicity behavior.  
To deal with this problem, we refer to a dichotomy by Bonatti, Diaz and Pujals (\cite{BDP}, or see section 2 in this paper for the precise statement) between the existence of dominated splitting and the existence of elliptic periodic points.   Due to the flexibility of  $C^1$ topology, there are sufficient spaces to locally embed a system with  nonuniform hyperbolicity.   The existence of elliptic periodic points allows us to do embeddings  on small invariant disks.   To paste this local map with the original map together, we prove a pasting lemma (Lemma \ref{pasting conservative2}).  

At the end of this section, we would like to address a weaker result for symplectic diffeomorphisms. In \cite{BB}, Baraviera and Bonatti  do perturbations  on some invariant bundles (not all bundles) of a dominated splitting by decreasing the  positive  integrated  Lyapunov exponent and at the same time increasing the negative one. Roughly speaking, the difference between the two variations should be  the new  integrated Lyapunov exponent of the center bundle. However, for symplectic diffeomorphisms, one can not do  this  kind of perturbation since the difference should always be zero. Thus, we can not ``borrow" hyperbolicity from the stable bundles and the unstable bundles to the center bundles.

Let $M$ be a $2d$-dimensional compact connected Riemannian manifold and $\omega$ be a symplectic form on $M$, i.e. a non-degenerate closed $2$-form. Taking $d$
times the wedge product of $\omega$ with itself we obtain a volume form on $M$. A $C^1$
diffeomorphism $f$ of $M$ is called symplectic if it preserves the symplectic
form, $f^*\omega=\omega$. Denote by $\mathcal{S}ym^r_{\omega}(M)$ the set of all $C^r$ symplectic diffeomorphisms ($r\geq 1$) on $M$.

 \begin{Thm}\label{Symplectic}Let $f$ be a $C^{ r}$  symplectic diffeomorphism ($r\geq1$) on a compact smooth  Riemannian manifold $M$. Then $C^1$ arbitrarily close to $f$ in $\text{Sym}^{ r}_{\omega}(M)$, there is a diffeomorphism $g\in\text{Sym}^{ r}_{\omega}(M)$ with non-zero Lyapunov exponents on a set of positive volume.  \end{Thm}
  \begin{Cor}\label{Density} Let  $\text{Sym}^+_{\omega}(M)$ be the set of $C^{ r}$ symplectic diffeomorphisms ($r\geq1$) on  a compact smooth  Riemannian manifold $M$ with positive metric entropy with respect to the symplectic form $\omega$. Then, we have $$\overline{\text{Sym}^+_{\omega}(M)}=\text{Sym}^1_{\omega}(M).$$
 \end{Cor}\smallskip
 
\section{Preliminary}
 \subsection{Dominated splitting and elliptic periodic points}
 As mentioned in the introduction,  to say something about the perturbation of  the systems with no domination, we shall use a dichotomy between the existence of dominated splitting and the existence of elliptic periodic points in this section.  
\begin{Def}A $Df-$invariant splitting $TM=E^1\oplus\cdots\oplus E^k$ is called
a {\em dominated splitting} if each $E^i$ is a continuous
$Df-$invariant subbundle of $TM$ and if there is some integer $n>0$
such that, for any $x\in M$, any $i<j$ and any non-zero vectors
$u\in E^i(x)$ and $\nu\in E^j(x)$, one has
$$\frac{\|Df^n(u)\|}{\|u\|}<\frac{1}{2}\frac{\|Df^n(\nu)\|}{\|\nu\|}.$$
 A map $f$ is called {\em partially hyperbolic} if there exists a dominated splitting $TM
= E^u\oplus E^c\oplus E^s$, into nonzero bundles such that,
for some Riemannian metric $\|\cdot\|$ on M, we have $$\|(Df
|E^u(x))^{-1}\|^{-1} > \|Df |E^c(x)\|\geq\|(Df |E^c(x))^{-1}\|^{-1}
> \|Df |E^s(x)\|$$ for every $x\in M$, where $E^s$ and $E^u$ denote the strong expanding and the strong contracting
invariant bundles, respectively. Such a splitting is automatically
continuous, 
\end{Def}
  \begin{Def} If for a periodic point $p$ of period $P(p)$, the tangent map $Df^{P(p)}(p)=Id$,
then we say that $p$ is {\em an elliptic periodic point}. 
\end{Def}

 Newhouse (\cite{N}) started  the question that when can we get elliptic periodic points for conservative diffeomorphisms.
 He  tacked with the conservative diffeomorphisms on surface and proved that for $C^1$ generic conservative diffeomoprhisms on surface,  either they are Anosov or have a dense set of elliptic periodic points. 
  Then, Sagin and Xia (\cite{SX}) proved that for $C^1$ generic symplectic diffeomorphisms, either they are partially hyperbolic or they have a dense set of elliptic periodic points (higher dimension case).   
 \begin{Lem}[Theorem 1, \cite{SX}]\label{LemSX}
There exists an open dense subset $\mathcal{U}$ of $\mathcal{S}ym^1_{\omega}(M)$ such that any diffeomorphisms
in $\mathcal{U}$ is either partially hyperbolic or it has an elliptic periodic point. There
exists a residual subset $\mathcal{R}$ of $\mathcal{S}ym^1_{\omega}(M)$ such that any diffeomorphisms in $\mathcal{R}$ is either partially
hyperbolic or the set of elliptic periodic points is dense on the manifold.
\end{Lem}
\medskip
 Herman also considered  the relation between the existence of dominated splitting and the existence of elliptic periodic points in \cite{H}:
 \begin{Con}[Herman]\label{H}\cite{H} Let $f\in \Diff^1_m(M)$ be a conservative diffeomorphism of a compact manifold $M$. Assume that there is a neighborhood  $\mathcal{U}$ of $f$ in $\Diff^1_m(M)$ such that for any $g\in \mathcal{U}$ and every periodic orbit $x$ of $g$ the matrix has at least one eigenvalue of modulus different from one. Then $f$ admits a dominated splitting.  
 \end{Con}
 Now let's state the dichotomy by Bonatti, Diaz and Pujals (\cite{BDP}). They proved Conjecture \ref{H}  under the assumption that $f$ is transitive. 
 
 \begin{Thm}[Corollary 0.4, \cite{BDP}]\label{BDP}Let $f\in \Diff^1_m(M)$ be a conservative transitive diffeomorphism of an $N$-dimensional manifold $M$. Then there is $l\in\mathbb{N}$ such that. 
 \begin{enumerate}
 \item either there is a conservative $\delta$ $C^1$ perturbation $g$ of $f$ having a periodic point $x$ of period $n\in\mathbb{N}$ such that $Dg^n(x)=Id$. 
 
 \item or $f$ admits a dominated splitting. 
 \end{enumerate}
 \end{Thm} 
 
  Thus for transitive conservative diffeomorphisms $f$,  no dominated splitting implies the existence of elliptic periodic points for some conservative diffeomorphisms $g\in \Diff^1_m(M)$ that are arbitrarily close to $f$.    
  On the other hand, transitivity is prevalence in conservative systems. Bonatti and Crovisier proved (\cite{BC}) that topological transitivity (i.e., existence of a dense orbit) is a generic property in $\Diff^1_m(M)$ (the topological mixing also holds (\cite{AC})).  
  \begin{Thm}\cite{BC}\label{BC}There exists a $C^1$ generic subset $R\subset \Diff^1_m(M)$ such that for any diffeomorphism $f\in R$, $f$ is topologically transitive. 
  \end{Thm}
   \subsection{Pasting lemma} 
   Starting with two conservative (or symplectic) diffeomorphisms (or flows), on $M$ and on some open set $V$ of $M$ respectively, we can ``paste" them together by pasting lemma and get a new conservative (symplectic) system. This new map coincide with one given map in an open set $U\subset V$  and with the other outside $V$.  Here we refer to a pasting lemma  for symplectic diffeomorphisms in \cite{AM}.
   \begin{Lem}[Lemma 3.9, \cite{AM}]\label{pasting symplectic}If $f$ is a $C^r$-symplectic diffeomorphism ($r \geq 1$) with $p$ as a periodic point  and $g$ is a local diffeomorphism ($C^r$-close to $f$) defined in a small neighborhood $V$ of the $f$-orbit of $p$, then there exists a $C^r$-symplectic diffeomorphism $h$ ($C^r$-close to $f$) and some neighborhood $U\subset V$ of $p$ satisfying $h\vert U= g$ and $g\vert V^c = f$ .   \end{Lem}
 In this paper, we need not only Lemma \ref{pasting symplectic} but also a conservative version of the Pasting Lemma. First we refer to a result by Dacorogna and Moser (\cite{DM}, Theorem 1), which will play a crucial role in the proof of  Pasting Lemma \ref{pasting conservative2}. 
 \begin{Lem}[\cite{DM}, Theorem 1]\label{Equation} Let $\omega$ be a compact manifold with $C^{\infty}$ boundary. Let $f, g\in C^{r}(\Omega), r>0$ be such that $f, g > 0$. Then there exists a $C^{r+1}$ diffeomorphism $\phi$ (with the same regularity at the boundary), such that,
 \[\begin{cases}g(\phi(x))\det(D\phi(x))= \lambda f(x), & x \in \Omega\\
 \phi(x) = x, & x \in \partial\Omega
 \end{cases}\]
where $\lambda= \frac{\int g }{\int f}.$ 
 Furthermore, there exists $C=C(r,\Omega)>0$ such that $\|\phi-Id\|_{C^{r+1}}\leq C\|f-g\|_{C^{r}}$. Also if $f, g$ are $C^{\infty}$, then $\phi$ is $C^{\infty}$.
  \end{Lem}
  Lemma \ref{pasting symplectic}, the pasting lemma for symplectic diffeomorphism,  is established based on the generated functions which are defined  locally.   Since we do not have generated function for conservative diffeomorphisms,  the proof of the pasting lemma for the conservative diffeomorphisms should be different from the proof for the symplectic case.  Lemma \ref{pasting conservative2} is similar to Theorem 3.6 in \cite{AM} with the interpolation of $g$ by $Df$ and the interpolation of the $C^1$ distance by the $C^r$ distance. 
 \begin{Lem}[Pasting Lemma for volume-preserving diffeomorphisms]\label{pasting conservative2}If $f$ is a $C^r$-conservative diffeomorphism ($r>1$) on a compact Riemannian manifold $M$ and $g$ is a $C^r$ local diffeomorphism ($C^1$-close to $f$) defined in a small neighborhood $V\subset M$ of  a point $x\in M$, then there exists a $C^r$-conservative diffeomorphism $G$ ($C^1$-close to $f$) and some neighborhood $U\subset V$ of $x$ satisfying $G\vert U = g$ and $G\vert V^c = f.$ 
 \end{Lem}  
 \begin{proof} As long as the neighborhood $V$ is small enough, there exists a local chart $\psi: V\rightarrow B(0,r_0)$ at $x$. With respect to this coordinate, we have the local map 
 $\psi \circ f\circ \psi^{-1}: B(0,r_0)\rightarrow \mathbb{R}^{n}$, where $n$ is the dimension of $M$,  with $\psi\circ f\circ \psi^{-1}(0)=0$. 
  
  Now let us consider a $C^1$ perturbation $h$ of $f$ associated with a $C^r$ bump function. Denote by $U=\psi^{-1}(B(0,r_0))$ and  $h(y) = \rho(y)g(y)+(1-\rho(y))f(y)$ (in local charts), where $\rho$ is a $C^{r}$ bump function such that $\rho\vert \psi^{-1}(B(0,r_0/2))=1, \rho\vert \psi^{-1}( B(0,r_0)^c)=0$ and $|\nabla\rho|\leq \frac{C}{r_0}$.  Now, we note that $\|h-f\|_{C^1}\leq C\cdot \|f-g\|_{C^1}$, where $C$ is a constant.
  Then if we denote by $\theta(y)=\det Dh(y)$, by the previous calculation, we obtain that $\theta$ is $C\cdot \|f-g\|_{C^1}$-$C^0$-close to $1$. Hence we have the same bound for the function $\hat{\theta}(y) = (\det Dh(h^{-1}(y)))^{-1}$.
  So, applying Lemma \ref{Equation} for the $C^{r-1}$ functions $\hat{\theta}$ and $1$ on the domain $\Omega=\overline{\psi^{-1}(B(0, r_0)\backslash B(0, r_0/2))}$, 
 we obtain that there exists a $C^{r}$  diffeomorphism $\xi$ that is a solution to the equation of Lemma \ref{Equation} which is $C^r$-close to the identity and is regular at the boundary of $\Omega$ where $\xi=Id$. Now observe that $G= \xi \circ h\in \Diff^r(M)$  is $C^1$ close to $f$ provided that $r_1$ is small and $g$ is $C^1$ close to $f$ on $U$. Since $G(y) =g(y)$ (in local charts) in $U$ we have that $\det DG(y)=1$ there, and in $ V^c$ we have $G(y)= f(y)$ so also $\det DG(y)=1$; finally, in $\Omega$ we have that $\det DG(x) = \det D\xi(h(x))Dh(x) = \det D\xi(y)Dh(h^{-1}(y)) = 1$, so $G$ is conservative.  
   Now we complete the proof. 
\end{proof}
   \subsection{Generic robustness of non-uniformly hyperbolic block}
 At the end of this section, we refer to a result saying the robustness of nonuniformly hyperbolic block (formulated in Sublemma 5.2 \cite{AB}). For any conservative diffeomorphism $f$, let $\text{NUH}(f)$ be the set of points without zero Lyapunov exponents. 
   \begin{Lem}\label{Robust block}
 There exists a $C^1$ generic subset $R\in\Diff^1_m(M)$ satisfying the following conditions:  for any $f\in R$, assume $\Lambda\subset \text{NUH}(f)$ to be a positive volume Borel $f$-invariant set with a dominated splitting. Then for any $\varepsilon> 0$,  for any $g\in\Diff^{1}_m(M)$ sufficiently $C^1$-close to $f$ it follows
  $m(\Lambda\backslash \text{NUH}(g))<\varepsilon.$
   \end{Lem}

 \section*{Acknowledgement}We would like to thank Jairo Bochi, Huyi Hu and Zhihong Xia for their useful suggestions.  The second author also would like to thank Zhihong Xia for providing office to her during Summer, 2015. 
 \section{The proof of Theorem \ref{Main}.}
  Before running into the proof of  Theorem \ref{Main}, we first present two Propositions.  
 \begin{Prop}\label{Prop1}Let $D^n$ be the unit ball in $\mathbb{R}^n$ and $m$ be the volume measure on $D^n$.  Then, for any $\delta>0$ and the  $C^{ r}$ nonuniformly hyperbolic diffeomorphism $f\in \Diff^1_m(D^n)$ constructed in Theorem \ref{DP}, there exists $C^{ r}$ diffeomorphisms $\{f_k\}_{k=1}^N$ on $D^n$  such that: 
 \begin{enumerate}
 \item $f_k$ preserves the volume measure $m$;
 \item  $\| f_k-\text{Id }\|_{C^1}\leq \delta$;
 \item $f=f_N\circ f_{N-1}\circ \cdots\circ f_1$. 
  \end{enumerate}
 \end{Prop}
  \begin{proof} Now let us first briefly recall the construction of maps in \cite{K} and prove this proposition for the Katok map.  
  
  {\bf The Katok map:} Let $g_0$ be a hyperbolic automorphism of the $2$-torus $T^2$ which has four fixed points $x_1=(0,0), x_2=(\frac{1}{2}, 0), x_3=(0,\frac{1}{2}),x_4=(\frac{1}{2}, \frac{1}{2})$
  (for example, the automorphism generated by the matrix 
  $\begin{bmatrix}
5 & 8           \\[0.3em]
8 & 13         \\[0.3em]
 \end{bmatrix}$ is appropriate). The desired diffeomorphism $g$ is constructed via the following commutative diagram 

\[
\xymatrix{T^2 \ar[r]^{\phi} \ar[d]^{g_0}
& T^2 \ar[d]^{g_1} \ar[r]^{\phi_1}& T^2 \ar[r]^{\phi_2}\ar[d]^{g_2}&S^2\ar[r]^{\phi_3}\ar[d]^{g_3}&D^2\ar[d]^{g}\\
T^2\ar[r]^{\phi}
& T^2\ar[r]^{\phi_1} &T^2\ar[r]^{\phi_2}&S^2\ar[r]^{\phi_3}&D^2}
\]
where $S^2$ is the unit sphere. The map $g_1$ is obtained by slowing down $g_0$ near
 the points $x_i, 1\leq i\leq 4$. Its construction depends upon a real-valued function $\phi$ which is defined on the unit interval $[0,1]$ and has the following properties:
 \begin{enumerate}
 \item $\phi$ is $C^{ r}$ except for the point $0$;
 \item $\phi(0)=0$ and $\phi(u)=1$ for $u\geq r$ where $0<r<1$ is a number;
 \item $\phi'(u)\geq 0$;
 \item $\int^1_0\frac{du}{\phi(u)}< r$.
 \end{enumerate} 
 with an extra condition on the function $\phi$ expressing a ``very slow" rate of convergence of the integral $\int^1_0\frac{du}{\phi(u)}$ near zero.   Since it is irrelevant to our proof, we skip the detail description about the ``very slow" condition. 
  
 
 Denote by $\tilde{g}^i_{\phi}$  the time one map generated by the vector field $v_{\phi}$ in $D^i_{r_0}, 1\leq i\leq 4$ given as follows:
 $$\dot{s}_1=(\log\alpha)s_1\phi(s_1^2+s_2^2), \dot{s}_2=-(\log\alpha)s_2\phi(s_1^2+s_2^2).$$ 
 Let $D=\cup^{4}_{i=1}D^i_{r_0}.$
 Therefore, the map 
 \[ g_1(x) = \begin{cases} g_0(x) & \quad \text{if }  x\in T^2\backslash D. \\ \tilde{g}^i_{\phi}(x) & \quad \text{if }  x\in D^i_{r_0}.\\ \end{cases} \]
defines a homeomorphism of the torus $T^2$ which is a $C^{ r}$ diffeomorphism everywhere except for the points $x_i, 1\leq i\leq 4$. 
  
  Once the maps $\phi_1,\phi_2,$ and $\phi_3$ are constructed the maps $g_2,g_3$ and $g$ are defined to make the above diagram commutative. We follow \cite{K} and describe a particular choice of maps $\phi_i, 1\leq i\leq 3$. 
  
  In  a neighborhood of each point $x_i, 1\leq i\leq 4$ the map $\phi_1$ is given by 
  $$\phi_1(s_1,s_2)=\frac{1}{\sqrt{k_0(s_1^2+s_2^2)}}\Big(\int^{s_1^2+s_2^2}_0\frac{du}{\phi(u)}\Big)^{\frac{1}{2}}(s_1,s_2)$$ 
  and it is identity in $T^2\backslash D$. 
  
  The map $\phi_2:T^2\rightarrow S^2$ is a double branched covering which is regular and $C^{ r}$ everywhere except for the points $x_i, 1\leq i\leq 4$ where it branches. There is a local coordinate system $(s_1,s_2)$ in a neighborhood of each points $p_i=\phi_2(x_i), 1\leq i\leq 4$ such that 
  $$\phi_2(s_1,s_2)=\big(\frac{s_1^2-s_2^2}{\sqrt{s_1^2+s_2^2}}, \frac{2s_1s_2}{\sqrt{s_1^2+s_2^2}}\big).$$

  In a neighborhood of $p_4=\phi_2(x_4)$, the map $\phi_3$ is given by 
  $$\phi_3(\tau_1,\tau_2)=\Big( \frac{\tau_1\sqrt{1-\tau_1^2-\tau_2^2}}{\sqrt{\tau_1^2+\tau_2^2}},\frac{\tau_2\sqrt{1-\tau_1^2-\tau_2^2}}{\sqrt{\tau_1^2+\tau_2^2}}\Big)$$
  and it is extended to a $C^{ r}$ diffeomorphism $\phi_3$ between $S^2\backslash \{p_4\}$ and $\text{int} D^2$.

  This concludes the construction of the diffeomorphism $g$ given in Theorem \ref{DP}.
  The following lemma given in \cite{HPT} shows that the map $g$ is diffeotopic to the identity map.  As we only need part of the results in Proposition 4 of \cite{HPT}, we prefer to present the proof here. 
 \begin{Lem}There exists a map $G : D^2 \times [0, 1]\rightarrow D^2$ with the following properties:
 \begin{enumerate}
 \item  $G(x,t)$ is $C^{ r}$ in $(x,t)$;
\item  $G(\cdot,0)=\text{id}$ and $G(\cdot,1) = g$;
 \item  for any $t \in [0, 1]$ the map $G(\cdot, t): D^2\rightarrow D^2$ is an area-preserving diffeomorphism.
 \end{enumerate}
 \end{Lem}
 \begin{proof}
 {\bf Step 1:} First of all, we shall prove that there exist small neighborhoods $U\subset V$ of $\partial D^2$ and  a vector field $F$ in $V$ which generates
an area-preserving flow $g^t : U \rightarrow D^2, -2 < t < 2$ for which $g\vert U = g^1$.  
 
 According to the construction of $g_1$,  $g_1\vert_{D^i_{r_0}}, i=1,2,3,4$ are the time one maps of the vector field $v_{\phi}$ given by:
 $$\dot{s}_1=(\log\alpha)s_1\phi(s_1^2+s_2^2), \dot{s}_2=-(\log\alpha)s_2\phi(s_1^2+s_2^2).$$    Since $g_2=\phi_1\circ g_1\circ \phi_1^{-1}$ and $\phi_1$ is differentiable, 
 we have $g_2$ is a time-one map for the vector field $v_{\phi,\phi_1}=d\phi_1\circ v_{\phi}\circ \phi_1^{-1}$.   Notice that the map $\phi_2:T^2\rightarrow S^2$ is a double branched covering and is regular and $C^{ r}$ everywhere except for the points $x_i$ where it branches.  Take the vector field $v_{\phi,\phi_1,\phi_2}=d\phi_2\circ v_{\phi,\phi_1}\circ \phi_2^{-1}$.  We claim $v_{\phi,\phi_1,\phi_2}$ is well defined even though $\phi_2$ is a two-to-one map. Assume $x$ and $-x$ be to the two $\phi_2$-pre-images of the point $y$ near $p_i$, namely, $\phi_2(x)=\phi_2(-x)=y\in S^2$.  Observing that $v_{\phi}(-x)=-v_{\phi}(x)$ and $\phi_1(-s_1,-s_2)=-\phi(s_1,s_2)$, we have that $$v_{\phi,\phi_1}(-x)=-v_{\phi,\phi_1}(x).$$  Thus, we have
 \begin{eqnarray*} v_{\phi,\phi_1,\phi_2}(y)&=&d\phi_2\circ v_{\phi,\phi_1}\circ \phi_2^{-1}(y)\\
 &=&d\phi_2(-x)\cdot v_{\phi,\phi_1}(-x)\cdot (\phi_2(-x))^{-1}\\
 &=&d\phi_2(x)\cdot v_{\phi,\phi_1}(x)\cdot (\phi_2(x))^{-1}.
 \end{eqnarray*}
 Thus, we finish the proof of the claim.  Similarly, define a new vector field $v_{\phi,\phi_1,\phi_2,\phi_3}$ around  $\partial D^2$ by   $v_{\phi,\phi_1,\phi_2,\phi_3}=d\phi_3\circ v_{\phi,\phi_1,\phi_2}\circ \phi_3$.    It turns out that $g\vert U$ is the time one map of the flow $g_t:V\rightarrow D^2$ generated by the vector field $F=v_{\phi,\phi_1,\phi_2,\phi_3}.$ 
 
{\bf Step 2:} We extend the vector field $F$ in the neighborhood $U$ of $\partial D^2$ to a vector field $\hat{F}$ on the whole disk $D^2$. Let $\hat{g}: D^2\rightarrow D^2$ be the time one map of the flow $\hat{F}$.  It follows that $\hat{g}^t\vert U=g\vert U$. Thus, we can define a new map $\tilde{g}=\hat{g}^{-1}\circ g$. Then, $\tilde{g}\vert U=\text{Id}\vert U$. We need a  result of Smale (see \cite{S}, Theorem B) saying that  the space of $C^{ r}$ diffeomorphisms of the unit disk which are equal to the identity in some neighborhood of the boundary  is contractible  with the $C^r$ topology, $1< r \leq \infty$. 
Applying this result to the diffeomorphism $\tilde{g}=\hat{g}\circ g^{-1}$, which is equal to the identity on $U$, we obtain a
homotopy $\tilde{G} : D^2 × [0, 1] \rightarrow D^2$ such that $\tilde{G}(\cdot, 0)= \text{id} \vert D^2$ and $\tilde{G}(\cdot, 1) = \tilde{g}$. Moreover, $\tilde{G}$ is $C^{ r}$ in $ (x, t)$. Hence $\tilde{G}$ is a diffeotopy from $\text{id}$ to $\tilde{g}$. Therefore, for each $t \in [0, 1]$, there is a neighborhood $U_t$ of $\partial D^2$ such that $\tilde{G}(\cdot, t)\vert U_t = \text{id}\vert U_t$. One can show that the set
$U = \text{int} U_t,  t\in[0,1]$ is not empty and is a neighborhood of $\partial D^2$. Denote $\tilde{g}^t = \tilde{G}(\cdot,t)$. It follows that $G(\cdot, t)= \hat{g}\circ \tilde{g}^t$ satisfies what we want.
   \end{proof}
 Then, $g$ can be written as $g=\tilde{g}_{N(\delta)}\circ\cdots\circ \tilde{g}_1$.  At this point, we finish the proof of Proposition \ref{Prop1} for the Katok map.   
 Now we shall outline Brin's construction from \cite{Brin}. 
 
 {\bf The Brin map:} Given a positive integer $n\geq 5$, set $k=[\frac{n-3}{2}]$ and consider the $(n-3)\times (n-3)$ block diagnal matrix $A=(A_i)$, where $A_i=\begin{bmatrix}
 2&1\\
 1&1
 \end{bmatrix}
 $ for $i<k$ and 
 \[A_k=\begin{cases}\begin{bmatrix} 2&1 \\ 1&1\\\end{bmatrix} \text{ if } n \text{ is odd. } \\[1mm]
 \begin{bmatrix}2&1&1\\1&1&1\\0&1&2\\
 \end{bmatrix} \text{ if } n \text{ is even. }
 \end{cases}
 \]
 Let $T^t$ be the suspension flow over $A$ with the roof function $H=H_0+\varepsilon H(x),$ where $H_0$ is a constant and the function $H(x)$ is such  that $|H(x)|\leq 1.$
 The flow $T^t$ is an Anosov flow on the phase space $\mathcal{Y}^{n-2}$ which is diffeomorphic to the product $T^{n-3}\times [0,1]$, where the tori $T^{n-3}\times 0$  and $T^{n-3}\times 1$ are identified by the action of $A$. One can choose the function $H(x)$ such that the flow $T^t$ has the accessibility property. Consider the following skew product map $R$ of the manifold $M=D^2\times \mathcal{Y}^{n-2}$ 
 $$R(z)=R(x,y)=(g(z), T^{\alpha(x)}(y)), \text{ for } z=(x,y),$$
 where $g$ is the Katok map constructed above and $\alpha:D^2\rightarrow \mathbb{R}$ is a non-negative $C^{ r}$ function which is equal to zero in the neighborhood $\mathcal{U}$ of the singularity set $\{q_1,q_2,q_3\}\cup \partial D^2$ where $q_i=\phi_3(p_i)$ and is strictly positive otherwise.  
 \begin{Lem}\cite{Brin}There exists a smooth embedding of the manifold $\mathcal{Y}^{n-2}$ into $\mathbb{R}^n$. 
 \end{Lem}
 So we can embed $R$ onto any compact smooth Riemannian manifold $K$ of dimension $n\geq 5.$ Since $T^t$ is a flow and $g=\tilde{g}_{N(\delta)}\circ\cdots\circ \tilde{g}_1$, we have $R=R_{N(\delta)}\circ\cdots\circ R_1$ where $R_i=(\tilde{g}_i, T^{\frac{i\alpha}{N(\delta)}})$ are $C^1$ close to identity.  
 For the Brin map, there exists a direction with zero Lyapunov exponent, i.e. the time direction in the suspension progress. Although there may not exist globally dominated splitting for Brin's map $R$,  Dolgopyat and Pesin \cite{DP} proved that they can perturb this map into a map with non-zero Lyapunov exponents. 
 \begin{Lem}[\cite{DP}] Given any $\varepsilon>0$, there is a $C^{ r}$ diffeomorphism $P:M\rightarrow M$ such that 
 \begin{enumerate}
 \item $d_{C^1}(P,R)\leq \varepsilon$;
 \item $P$ is ergodic and invariant with respect to a smooth measure;
 \item $P$ has only non-zero Lyapunov exponents. 
 \end{enumerate}
 \end{Lem}
 So, we can also write $P=R_{N(\delta)}\circ\cdots\circ R_1 \circ P_0$ where $P_0$ is also $C^1$ close to identity.  We leave the cases when $\dim M=3$ and $\dim M=4$ to the reader. Thus, we finish the proof of Proposition \ref{Prop1}.
  \end{proof}
 
 \begin{Prop}\label{Prop}Let $M$ be a smooth connected Riemannian manifold with the volume measure $m$. Every $C^1$ conservative diffeomorphims $\Diff^1_m(M)$ with elliptic periodic points can be arbitrarily approximated by a $C^{ r}$ conservative diffeomorphism with non-zero Lyapunov exponents on a set with positive volume. 
 \end{Prop}
     
 \begin{proof}
 Assume $f$ to be the $C^1$ conservative diffeomorphism  with the elliptic periodic point $p$. Let
 $P(p)$ be the period of the periodic point $p$. 
 
 {\bf Step 1:} By a $C^1$ small 
perturbation, for any small number $\varepsilon>0$, there exists a small  ball (under the local canonical coordinates) $D_1$ in
$M$, an integer $k_1$ and a $C^r$ conservative diffeomorphism $g_1\in\Diff^1_m(M) $, such that 
\begin{enumerate}
\item $g_1^{k_1P(p)}|_{D_1}=Id$;
\item $g_1^{k_1i}(D_1)\cap g_1^{k_1j}(D_1)=\emptyset,$ for $i\neq j\in [0,P(p)-1]$;
\item  $d_{C^1}(f,g_1)\leq \varepsilon$.  
\end{enumerate}
The map $g_1$ can be obtained through the local linearization of the diffeomorphism $f$ at the elliptic periodic point $p$.  Fix a sufficiently small number $\varepsilon>0$.  Denote by $\tilde{D}_1=\cup_{i=0}^{P(p)-1} g_1^{k_1i}(D_1).$

{\bf Step 2:} Now let us do another $C^1$  small perturbation to get any periodic disk with large period.  For any small number $\frac{1}{k}\leq \varepsilon$, where $k\in \mathbb{N}^+$, consider the rotation $R_{\frac{2\pi}{kk_1}}$ with rotation angle $\frac{2\pi}{kk_1}.$ It is easy to see that $R_{\frac{2\pi}{kk_1}}(D_1)=D_1.$
There exists a $C^r$ conservative diffeomorphism $g_2\in\Diff^1_m(M)$, such that
\begin{enumerate}
\item $g_2(x)=g_1\circ R_{\frac{2\pi}{kk_1}}(x),$ if $x\in D_1$;
\item $g_2(x)=g_1(x),$ if $x\in g^{i}(D_1)$ for $i\in[1,P(p)-1]\cap\mathbb{Z}$ ;
\item $d_{C^1}(g_1,g_2)\leq \varepsilon$. 
 \end{enumerate} 
 Thus, we have $g_2^{kk_1P(p)}|_{D_1}=Id$ and a smaller disk $D_2\subset D_1$ such that $g_2^{i}(D_2)\cap g_2^j(D_2)=\emptyset,$ for any $i\neq j\in [0,P(p)kk_1-1]\cap\mathbb{Z}.$    Denote $\tilde{D}_2=\cup^{kk_1P(p)-1}_{i=1}g_2^i(D_2).$

{\bf Step 3:} Let us perturb $g_2$  to  embed  maps  with non-uniformly hyperbolicity onto $\tilde{D}_2$.  Let $f$ be the conservative map constructed in Theorem \ref{DP} for $D_2$.  By Proposition \ref{Prop1},  there exists $C^{r}$ conservative diffeomorphisms $\{f_i\}_{i=1}^N$ on $D_2$  such that  $\| f_i-\text{Id }\|_{C^1}\leq \varepsilon$ and  $g=f_N\circ f_{N-1}\circ \cdots\circ f_1$ is hyperbolic and ergodic on $D_2$. Without loss of generalization, we assume that $N=kk_1$.  
 Then, there exists a $C^{ r}$ conservative diffeomorphism $g_3\in \Diff^1_m(M)$, such that 
\begin{enumerate}
\item $g_3(x)=g_2\circ R_{\frac{2i\pi}{kk_1}}\circ f_i \circ R^{-1}_{\frac{2i\pi}{kk_1}}(x)$ if $x\in g_2^{iP(p)}(D_2)$;\smallskip
\item $g_3(x)=g_2(x),$ if $x\in g_2^{(i+1)P(p)+j}(D_2), j\in [1,P(p)-1]\cap\mathbb{Z}$;\smallskip
\item $d_{C^1}(g_3,g_2)\leq \varepsilon$,
 \end{enumerate} 
where $i\in[0,kk_1-1]\cap\mathbb{Z}.$ Thus, $g_3^{i}=R_{\frac{2i\pi}{kk_1}}\circ f_i\circ \cdots\circ f_1,$ for any $x\in D_2$ and $i\in[1,kk_1]\cap\mathbb{Z}$.  Moreover, $g_3^{P(p)kk_1}(x)=g(x),$ for any $x\in D_2.$  
The map $g_3$ is the one we want. We finish the proof of Proposition \ref{Prop}.   
\end{proof}
Now we are ready to present the proof of our main Theorem \ref{Main}:
\begin{proof}[The proof of Theorem \ref{Main}:]
Denote by $\Diff^{\sharp}(M)$ the maps in $\Diff^1_m(M)$ with global dominated splitting.  According to Theorem \ref{ACWT} and Corollary \ref{ACW}, we obtain the $C^1$ generic maps  in $\Diff^{\sharp}(M)$ are nonuniformly Anosov.  Thus, it follows from Lemma \ref{Robust block} that there exists a $C^1$ dense and open subset  in $\Diff^{\sharp}(M)$ such that $m(\text{NUH}(f))$ is positive.  On the other hand, by the $C^1$ density of $C^{ r}$ conservative diffeomorphisms in $\Diff^1_m(M)$ (\cite{A2}), we have  the maps in $\Diff^{ r}_m(M)$ with $m(\text{NUH}(f))>0$ are $C^1$ dense in $\overline{\Diff^{\sharp}(M)}$. 

  It remains to prove the maps without zero Lyapunov exponents on a positive volume subset are also dense in $\Diff^1_m(M)\backslash \overline{\Diff^{\sharp}(M)}$.  Assume $f$ to be a map in $\Diff^1_m(M)\backslash \overline{\Diff^{\sharp}(M)}$.  By Theorem \ref{BC}, we can assume that $f$ is  transitive.  Applying Theorem \ref{BDP}, we obtain $g_n\rightarrow f$ such that $g_n$ has elliptic periodic point $p_n$ for any $n\in \mathbb{N}$.  This implies $D^{P(p_n)}g_n|_{p_n}=Id$. Thus, combining Proposition \ref{Prop1} and Proposition \ref{Prop}, we can do perturbation to $g_n$ around the elliptic periodic points $p_n$ such that the $\tilde{g}_n\in \Diff^{ r}_m(M)$ is nonuniformly hyperbolic around $p_n$ and with nonuniformly hyperbolic behavior and positive entropy. Moreover, by  Pasting Lemma \ref{pasting conservative2} of conservative maps, we can extend $g_3$ to the whole manifold $M$. Now we complete the proof of Theorem \ref{Main}.
 \end{proof}

\section{The proof of Theorem \ref{Symplectic}}

\begin{Lem}\label{Lem1}
  Let $f\in \mathcal{S}ym^1_{\omega}(M)$ and $p$ be an elliptic periodic point of $f$. Then there is a $C^1$ arbitrarily small perturbation $g\in \mathcal{S}ym^{ r}_{\omega}(M)$ of $f$ without zero Lyapunov exponents on a positive measure set. 
\end{Lem}

\begin{proof}Let $P(p)$ be the period of the periodic point $p$. By a small
perturbation, we can assume that $D^nf$ is a rational rotation. Then
there exists a positive integer $k_1P(p)>0$ and a measurable set $D_1$ in
$M$, such that $f^{k_1P(p)}|_{D_1}=Id$. Following the second steps in the proof of Proposition \ref{Prop}, for any small number $\varepsilon>0$, we obtain $D_2$ and any integer $k$ large enough such that  there exists a small perturbation $g_2$ of $f$ such that 
$g_2^{kk_1P(p)}\vert D_2=Id$, $d_{C^1}(g_2, f)\leq \varepsilon$ and $g_2^{i}(D_2)\cap g_2^j(D_2)=\emptyset,$ for any $i\neq j\in [0,P(p)kk_1-1]\cap\mathbb{Z}.$ 

Let us perturb $g_2$  to  embed  a map  with non-uniformly hyperbolicity behavior  onto $D_2$. We can write $D_2=D^2\times \cdots \times D^2$.  Take $g(x)=\underbrace{h\times\cdots\times h}_{d \mbox{ times}}$ on $D_2$, where $h$ is the Katok map in  \cite{K} (or see the proof of Proposition \ref{Prop1}) on $D^2$. Hence $g$ is a symplectic diffeomorphisms without zero exponents  on $D_2$. However it is not a small
perturbation of $Id$. According to Proposition \ref{Prop1}, for any $\varepsilon>0$ small enough, there exists a sequence of small conservative perturbations
$\{h_i\}^{N}_{i=1}$ such that
$$h_N\circ\cdots\circ h_0=h$$
and $d_{C^1}(h_i,Id)\leq \varepsilon.$
Thus it follows that 
\begin{eqnarray*}
g&=&h_N\circ\cdots\circ h_0\times \cdots \times h_N\circ\cdots\circ h_0\\
&=& (h_N,\cdots, h_N)\circ \cdots \circ (h_0,\cdots,h_0)
\end{eqnarray*}
Without loss of generalization, we assume  $N=kk_1$.   Consider the rotation $R_{\frac{2\pi}{kk_1}}$  on $D^2$ with rotation angle $\frac{2\pi}{kk_1}.$ Then, there exists a $C^{ r}$ conservative diffeomorphism $g_3\in \text{Sym}^r_{\omega}(M)$, such that 
\begin{enumerate}
\item $g_3(x)=g_2\circ (R_{\frac{2i\pi}{kk_1}}, \cdots, R_{\frac{2i\pi}{kk_1}})\circ (h_i,\cdots, h_i) \circ (R^{-1}_{\frac{2i\pi}{kk_1}}(x_1), \cdots, R^{-1}_{\frac{2i\pi}{kk_1}}(x_d))$ if $x\in g_2^{iP(p)}(D_2)$;\smallskip
\item $g_3(x)=g_2(x),$ if $x\in g_2^{(i+1)P(p)+j}(D_2), j\in [1,P(p)-1]$;\smallskip
\item $d_{C^1}(g_3,g_2)\leq \varepsilon$,
 \end{enumerate} 
where $i\in[0,kk_1-1]\cap\mathbb{Z}.$ Thus, $g_3^{i}=(R_{\frac{2i\pi}{kk_1}}, \cdots, R_{\frac{2i\pi}{kk_1}})\circ (h_i,\cdots, h_i)\circ \cdots(h_0,\cdots, h_0),$ for any $x\in D_2$ and $i\in[1,kk_1]\cap\mathbb{Z}$.  Moreover, $g_3^{P(p)kk_1}(x)=g(x),$ for any $x\in D_2.$  
Hence the map $g_3$ has no zero Lyapunov exponents on the positive measure set $D_2$ and  is $C^1$ close to $f$. 
Moreover, by  Pasting Lemma \ref{pasting symplectic} of symplectic maps, we can extend $g_3$ to the whole manifold $M$. Now we finish the proof. 
\end{proof}
\begin{proof}[The proof of Theorem \ref{Symplectic}]
By the robustness of partially hyperbolic diffeomorphisms and the the density of smooth symplectic maps in $C^1$ symplectic maps (\cite{Z}),  we obtain $C^{ r}$ partially hyperbolic  symplectic diffeomorphisms  is $C^1$ dense in the $C^1$ partially hyperbolic symplectic diffeomorphisms.   Denote by $\text{Sym}^{\sharp}_{\omega}(M)$ the $C^1$ partially hyperbolic symplectic diffeomorphisms. It suffices to tackle with the complementary subset $\text{Sym}^{\sharp}_{\omega}(M)^c$ of $\text{Sym}^{\sharp}_{\omega}(M)$ in the $C^1$ symplectic diffeomorphisms. 
Due to Lemma \ref{Lem1},  there is a  $C^1$ dense subset of diffeomorphisms in $\text{Sym}^{\sharp}_{\omega}(M)^c$ with elliptic periodic points. 
Moreover, applying Lemma \ref{LemSX},  we obtain that the $C^r$ symplectic diffeomorphisms with non-zero Lyapunov exponents on a set of positive volume are $C^1$ dense in $\text{Sym}^{\sharp}_{\omega}(M)^c$. Thus, we complete the proof.
\end{proof}

 \end{document}